\documentclass[10pt,a4paper]{amsart}	
\usepackage{amsfonts}			
\usepackage{amsmath}			
\usepackage{amsthm}			
\usepackage{amssymb}			
\usepackage{graphicx}			
\usepackage{enumerate}			
\usepackage{mathrsfs}			
\usepackage{comment}			
\usepackage{hyperref}
\usepackage{pstricks}			
\usepackage[all,knot,arc]{xy}		
\usepackage{rotating}
\usepackage{array}

\graphicspath{{../immagini/}}

\newcommand{\executeiffilenewer}[3]{%
\ifnum\pdfs

trcmp{\pdffilemoddate{#1}}%
{\pdffilemoddate{#2}}>0%
{\immediate\write18{#3}}\fi%
}

\theoremstyle{plain}
\newtheorem{thm}{Theorem}[section]
\newtheorem{lem}[thm]{Lemma}
\newtheorem{prop}[thm]{Proposition}
\newtheorem{cor}[thm]{Corollary}

\theoremstyle{definition}
\newtheorem{defi}[thm]{Definition}
\newtheorem{rmk}[thm]{Remark}
\newtheorem{es}[thm]{Example}

\newtheorem{quest}[thm]{Question}

\theoremstyle{plain}
\newtheorem{introthm}{Theorem}

\theoremstyle{definition}

\newcommand{\mc}[1]{\mathcal{#1}}

\newcommand{\on}[1]{\operatorname{#1}}

\newcommand{\wt}[1]{\widetilde{#1}}

\newcommand{\eps}{\varepsilon}

\newcommand{\nN}{\mathbb{N}}
\newcommand{\nZ}{\mathbb{Z}}
\newcommand{\nQ}{\mathbb{Q}}

\newcommand{\nC}{\mathbb{C}}

\newcommand{\nP}{\mathbb{P}}
\newcommand{\nK}{\mathbb{K}}

\newcommand{\nmat}[3]{{\mc{M}(#1 \times #2, #3)}}

\newcommand{\one}{{1\hspace*{-0.8 mm}\mathrm{l}}}

\newcommand{\ve}[1]{{#1}}			

\newcommand{\xx}{{x}}

\newcommand{\vxx}{\ve{\xx}}

\newcommand{\mi}[1]{{#1}}		

\newcommand{\lhs}{\mathbb{I}}
\newcommand{\rhs}{\mathbb{II}}

\newcommand{\abs}[1]{\left|#1\right|}

\newcommand{\norm}[1]{\left\|#1\right\|}

\newcommand{\comp}{}

\newcommand{\kronecker}[2]{\delta_{#1}^{#2}}

\newcommand{\bin}[2]
{\left(\!
\begin{array}{c}
#1\\
#2
\end{array}
\!
\right)}

\newcommand{\sumaa}[2]{\hspace{-#2}\sum_{#1}\hspace{-#2}}
\newcommand{\sumab}[3]{\hspace{-#2}\sum_{#1}\hspace{-#3}}

{\begin{proof}}
{\end{proof}}

\newenvironment{myabstract}[1]
{\begin{center}\small{\textbf{#1}}

\vspace{0.25cm}
\begin{minipage}[l]{10.4cm}\begin{small}}
{\end{small}\end{minipage}\end{center}}

\title{Classification of one-dimensional superattracting germs in positive characteristic}
\author[M. Ruggiero]{Matteo Ruggiero}
\thanks{\noindent Fondation Math\'ematique Jacques Hadamard (FMJH).\\
Centre de Math\'ematiques Laurent Schwartz, 
\'Ecole Polytechnique, 
91128 Palaiseau Cedex, France.\\
E-mail: ruggiero@math.polytechnique.fr, tel: (+33) (0)1 69 33 49 18.
\\
The author was supported by the ERC-starting grant project ``Nonarcomp'' no.307856.
}
\date{\today}

\makeindex

\begin{document}

\maketitle

\begin{myabstract}{Abstract}
We give a classification of superattracting germs in dimension $1$ over a complete normed algebraically closed field $\nK$ of positive characteristic up to conjugacy.
In particular we show that formal and analytic classifications coincide for these germs.
We also give a higher dimensional version of some of these results.
\end{myabstract}

\pagestyle{plain}				


\section*{Introduction}

Recent interest arose in understanding the local dynamics of analytic germs $f:(\nK^N,0) \to (\nK^N,0)$ over a (complete normed) field $\nK$ of positive characteristic.
One of the first works in this direction is given by \cite{herman-yoccoz:smalldivisornonarchi}, where the authors deal with some problems of small divisors and resonances in the non-archimedean setting, in particular for analytic germs $f:(\nK,0) \to (\nK,0)$ with $\abs{f'(0)}=1$.
If the characteristic of $\nK$ is zero, the authors show that $f$ is always analytically linearizable, as far as $\lambda:=f'(0)$ is not a root of unity.
This result does not hold in positive characteristic, the problem due to the presence of small divisors, that gives an obstruction to the convergence of the formal conjugacy between $f$ and its linear part $x \mapsto \lambda x$ (see \cite{lindahl:siegellinprimechar}).

In this paper, we are concerned with superattracting germs, characterized by the property that the differential $df_0$ at $0$ is nilpotent.

In dimension one, any superattracting germ can be written under the form $f(x)=Cx^d (1+ \eps(x))$ with $C \neq 0$, $d \geq 2$ and $\eps(0)=0$.
When $\nK$ is the field of complex numbers endowed with the standard euclidean norm, a classical result by B\"ottcher \cite{bottcher:1904} states that $f$ is analytically conjugate to the map $x \mapsto x^d$.
This result still holds for superattracting germs $f:(\nK,0) \to (\nK,0)$ over any field $\nK$ of characteristic zero, endowed with any (archimedean or non-archimedean) complete norm, provided that $\sqrt[d-1]{C} \in \nK$.

B\"ottcher theorem also holds over fields $\nK$ of characteristic $p > 0$ when $d$ and $p$ are coprime.
However it does not hold when $p$ divides $d$.

Consider for example the germs $F(x)=x^p$ (called the \emph{Frobenius automorphism}) and $f(x)=x^p(1+x)$.
Since $F'(x) \equiv 0$ while $f'(x)=x^p$, these two germs cannot be conjugate one to the other.

Since any germ $f$ whose derivative is identically zero can be factorized through the Frobenius automorphism, there exists a unique $m \in \nN$ for which $f=g \circ F^m$, where $g:(\nK,0)\to (\nK,0)$ satisfies $g' \not \equiv 0$.
Set $d=\on{ord}_0(g)$ and $r_0=1+ \on{ord}_0(g')-d$, where $\on{ord}_0$ denotes the order of vanishing at $0$.
The numbers $m$, $d$ and $r_0$ are invariants of (formal) conjugacy.
Notice that either $r_0=0$ (when $p$ and $d$ are coprime), or $r_0 > 0$ is coprime to $p$ (and $p$ divides $d$).

In this paper, we provide the formal and analytic classification of superattracting germs in dimension one, over any algebraically closed field $\nK$ of positive characteristic.

\begin{introthm}\label{thm:Bhard}
Let $\nK$ be an \emph{algebraically closed} field of characteristic $p > 0$.
Let $f:(\nK,0) \to (\nK,0)$ be a superattracting germ.
Then $f$ is formally conjugate to a map $\wt{f}:(\nK,0)\to(\nK,0)$ of the form
\begin{equation}\label{eqn:Bottchernormalform}
\wt{f}(x)=x^{dp^m}\Big(a\big(x^{p^{m+1}}\big)+bx^{r_0p^m}\Big),
\end{equation}
where:
\begin{itemize}
\item $m \in \nN$, $d \in \nN^*$ and $dp^m \geq 2$,
\item either $r_0=0$, or $r_0$ is coprime to $p$ when $p$ divides $d$,
\item $b \neq 0$, and $a \in \nK[z]$ is a polynomial of degree $< r_0/(p-1)$,
\item when $r=0$, then $a \equiv 0$ and $b=1$,
\item when $r>0$, then $a(0)=1$ and $b$ is uniquely determined up to the multiplication by a root of unity $\zeta$ such that $\zeta^{dp^m}=\zeta$.
\end{itemize}
\end{introthm}

Theorem \ref{thm:Bhard} does not provide a complete classification of superattracting germs in positive characteristic, since the polynomial $a$ is not uniquely determined.
In Theorem \ref{thm:Bext} we shall describe normal forms with the property that for any $f$ there exists a \emph{finite} number of such normal forms conjugate to it.

Special cases of our main result were known to the experts.
The case when $d$ and $p$ are coprime can be proved as in the classical case of B\"ottcher's theorem (see Theorem \ref{thm:Beasy}), and Gardner Spencer in his thesis \cite{gardnerspencer:phdthesis} gave a formal classification of superattracting germs when $m=0$ and $d=p$ (see Remark \ref{rmk:normalformse1N'}).

When $\nK$ is endowed with a (complete) norm, we show that analytic and formal conjugacy of superattracting germs in dimension one coincide.

\begin{introthm}\label{thm:Ecalle}
Let $\nK$ be an algebraically closed field of positive characteristic, endowed with a (complete) norm.
Suppose two germs $f, \wt{f}:(\nK,0) \to (\nK,0)$ are formally conjugate.
Then they are \emph{analytically} conjugate.
\end{introthm}
\'Ecalle conjectured that formal and analytic classifications coincide for superattracting germs $f:(\nC^N,0) \to (\nC^N,0)$ over the complex numbers (with standard euclidean norm), in any dimension $N\geq 1$.
Hence Theorem \ref{thm:Ecalle} gives a positive answer to \'Ecalle's conjecture in our setting.

\smallskip

As an example of result that holds for superattracting germs over fields of positive characteristic in higher dimensions, we give a sufficient condition for a superattracting germ to be conjugate to a monomial map.
\begin{introthm}\label{thm:monomial}
Let $\nK$ be a complete normed (possibly non algebraically closed) field of characteristic $p > 0$.
Let $\ve{f}:(\nK^N, 0) \rightarrow (\nK^N,0)$ be a superattracting germ of the form
\begin{equation}\label{eqn:Mstart}
\ve{f}(\vxx)=\ve{C} \vxx^\mi{D} \big(\one + \eps(\vxx)\big),
\end{equation}
where $\vxx \in \nK^N$, $\ve{C} \in (\nK^*)^N$, $\mi{D} \in \nmat{N}{N}{\nN}$ and $\eps:(\nK^N, 0) \to (\nK^N, 0)$ with $\ve{\eps}(\ve{0})=\ve{0}$.

Suppose $\det \mi{D}$ is coprime to $p$.

Then $f$ is analytically conjugate to its leading monomial part
\begin{equation}\label{eqn:Mnormalform}
\ve{\wt{f}}(\vxx) = \ve{C} \vxx^\mi{D}.
\end{equation}
\end{introthm}
Notice that in general a map of the form \eqref{eqn:Mstart} will not be conjugate to a map of the form \eqref{eqn:Mnormalform} with $\ve{C}=(1, \ldots, 1)=:\one$, not even when $\nK$ is algebraically closed.
Indeed, it is possible only when $1$ is not an eigenvalue for $\mi{D}$ (see Remark \ref{rmk:MwhataboutC}).

To read \eqref{eqn:Mstart} and \eqref{eqn:Mnormalform}, we used the following notation.
Write $\vxx=(\xx^1, \ldots, \xx^N)$ and $\mi{D}=(d_i^j)$.
Then we set $\vxx^\mi{D}=\big((\vxx^{\mi{D}})^1, \ldots, (\vxx^{\mi{D}})^N\big)$, with
$$
(\vxx^{\mi{D}})^j = \prod_{i=1}^N (\xx^i)^{d_i^j}
$$
The product between two vectors in $\nK^N$ is the product coordinate by coordinate: if $\ve{C}=(C^1, \ldots, C^N)$, then $\ve{C} \vxx \in \nK^N$ is defined by $\ve{C} \vxx = (C^1 \xx^1, \ldots, C^N \xx^N)$.

\medskip

We now indicate how we prove our stated results.
The proof of Theorem \ref{thm:Bhard} has a combinatorial flavour.
Indeed, we need to solve the conjugacy relation $\Phi \circ f = \wt{f} = \Phi$, where $f$ is the given germ, $\wt{f}$ is the candidate normal form, and $\Phi$ is the unknown change of coordinates.
We write $f$, $\wt{f}$ and $\Phi$ in formal power series.
The conjugacy relation induces an infinite number of relations between the coefficients of such formal power series, where the unknowns are the coefficients of $\Phi$ and $\wt{f}$.
We solve these equations by induction, that is at some points quite intricate.

To prove Theorem \ref{thm:Ecalle}, we estimate the growth of the coefficients of a conjugacy $\Phi$ between $f$ and $\wt{f}$ in the normal form given by Theorem \ref{thm:Bhard}.
The case when $d$ and $p$ are coprime is easy, and can be dealt with classical arguments (similar to the classical proof of B\"ottcher's theorem).
When $p$ divides $d$, the combinatorics is much more delicate to deal with.
In this case, we estimate the growth rate of the coefficients of $\Phi$ by majorant series techniques, using the recursion formulae derived in the proof of Theorem \ref{thm:Bhard}.

For both these results, the main difficulties arise from the delicate combinatorics of the equations to solve, given by the positive characteristic setting.

The proof of Theorem \ref{thm:monomial} is analogous to the one working in the complex setting. We use part of the techniques developed in \cite{ruggiero:rigidgerms} to prove the result.

\medskip

The paper is organized as follows.
In the first section we fix some notations and recall a few properties of non-archimedean norms over fields of positive characteristic.
In the second section we introduce the discrete invariants for superattracting germs in dimension one, and study their behavior under composition and iteration.
In the third section we state the formal classification Theorem \ref{thm:Bext} of superattracting germs in dimension one.
We also give some remarks, deduce Theorem \ref{thm:Bhard}, and give some restrictions on the invariants of superattracting germs given by the action at infinity of a polynomial in $\nK$.
In the fourth section we give the analytic classification of superattracting germs when $d$ and $p$ are coprime.
In the fifth section we prove Theorem \ref{thm:Bext}, and in the sixth section we prove Theorem \ref{thm:Ecalle}.
In the seventh (and last) section, we conclude by proving Theorem \ref{thm:monomial}, and by giving some remarks and open questions on the local classification of superattracting germs in higher dimensions over fields of positive characteristic.


\section{Basics}

In this section we recall a few properties for non-archimedean norms.
For the whole paper, all norms will be complete.

\begin{prop}\label{prop:convergenceiseasy}
Let $(a_n)_n$ be a sequence in a any field $\nK$ endowed with a non-archimedean norm.
Then
$$
\prod_n (1+a_n) \mbox{ converges} \Longleftrightarrow \sum_n a_n \mbox{ converges} \Longleftrightarrow a_n \to 0.
$$
\end{prop}
%
%

We introduce here the $p$-adic valuation on integers, that will be very useful throughout the whole paper.

\begin{defi}
Let $p$ be a prime number, and $n \in \nZ$. The \emph{$p$-adic valuation} $\nu_p$ is defined by
$$
\nu_p(n)=\sup\{k \in \nN\ |\ p^k \text{ divides } n\} \in \nN \cup \{+\infty\}.
$$
\end{defi}
From now on, $p$ will be always denote a prime number, equal to the characteristic of $\nK$.

\begin{rmk}\label{rmk:rootwelldefined}
Let $b \in \nN$ be such that $\nu_p(b)=0$, i.e., $p$ and $b$ are coprime.
Then $(1+\xx)^{1/b}$ is a well-defined analytic germ.
Indeed, we can define
\begin{equation}\label{eqn:defprimeroot}
(1+\xx)^{1/b}=\sum_{n=0}^\infty \bin{1/b}{n} \xx^n,
\end{equation}
where
$$
\bin{1/b}{n}= \frac{b^{-1} \cdot (b^{-1}-1) \cdots (b^{-1}-n+1)}{n!} = \frac{(1-b)\cdots (1-(n-1)b)}{n!b^n}.
$$
Set $b_k=1-bk$.
Notice that if $\nu_p(b_k)=v$, then $\nu_p(b_{k+h})< v$ for $0 < h < p^v$, and $\nu_p(b_{k+p^v})\geq v$.
It follows that
$$
\on{Card}\{k\ |\ 0 \leq k < n, \nu_p(b_k)\geq v\} \geq \left\lfloor\frac{n}{p^v}\right\rfloor
= \on{Card}\{k\ |\ 0 \leq k < n, \nu_p(k+1)\geq v\}.
$$
Hence $\nu_p(b_0 \cdots b_{k-1}) \geq \nu_p(n!)$ and $\bin{1/b}{n}$ is a well defined element in $\nK$, of norm either $0$ or $1$.
Therefore \eqref{eqn:defprimeroot} defines an analytic germ over $\nK$.

It follows that if $\eps:(\nK^d,0) \to (\nK,0)$ is an analytic germ with $\eps(\ve{0})=0$, we can define $(1+\eps(\vxx))^{a/b}$ as an analytic germ $(\nK^d,0) \to \nK$, as long as $\nu_p(a) \geq \nu_p(b)$.
\end{rmk}

We shall need the next proposition to study the convergence of formal power series.
\begin{prop}\label{prop:prodconverges}
Let $\nK$ be a normed field of characteristic $p>0$.
Let $(\ve{\eps}_n)_n$ be a sequence of $r$-uples of formal power series in $N$ variables.
Let $(\mi{D}_n)_n$ be a sequence of matrices in $\nmat{r}{s}{\nQ}$ (with $s \in \nN^*$), with entries that are of the form $a/b \in \nQ$ with $\nu_p(a) \geq \nu_p(b)$.
Suppose
$$
\norm{\ve{\eps}_n(\vxx) \mi{D}_n} \to 0
$$
for $n \to +\infty$ and $\norm{\vxx}$ small enough.
Then
\begin{equation}\label{eqn:prodconverges}
\prod_{n=0}^\infty \big(\one+\ve{\eps}_n(\vxx)\big)^{\mi{D}_n} < +\infty
\end{equation}
converges for $\norm{\vxx}$ small enough.
\end{prop}
\begin{proof}
First, notice that \eqref{eqn:prodconverges} is a vector expression. We can estimate each entry, and suppose $s=1$.
Write $\mi{D}_n=(d_n^1, \ldots, d_n^s)$, where $d_n^j$ are the columns of $\mi{D}_n$, and $\eps_n= (\eps_n^1, \ldots, \eps_n^r)$.
Then
$$
\big(\one + \ve{\eps}_n\big)^{\mi{D}_n} = \prod_{j=1}^r (1+\eps_n^j)^{d_n^j}.
$$
Since $d_n^j=a/b$ with $\nu_p(a)\geq \nu_p(b)$, then $(1+\eps_n^j)^{d_n^j}$ is well defined for any $j$ (see Remark \ref{rmk:rootwelldefined}).
The statement then follows by Proposition \ref{prop:convergenceiseasy}.
\end{proof}

A special role among superattracting germs in dimension one is played by the \emph{Frobenius automorphism} $F(x)=x^p$.
The next proposition shows what happens if we conjugate a germ $\Psi$ by $F$.
\begin{prop}\label{prop:defT}
Let $\Psi \in \nK[[x]]$ be any formal power series. Then there exists a formal power series $T\Psi \in \nK[[x]]$ such that
$$
F \circ \Psi = T\Psi \circ F,
$$
where $F$ denotes the Frobenius automorphism.

Suppose $\Phi=T\Psi$ and there exists $M>0$ such that $\abs{\Phi(x)} \leq M \abs{x}$.
Then $\abs{\Psi(x)} \leq M^{1/p} \abs{x}$.
\end{prop}
\begin{proof}
The operator $T:\nK[[x]] \to \nK[[x]]$ is defined by
\begin{equation}\label{eqn:defT}
T\left(\sum_n\psi_n x^n\right):=\sum_n \psi_n^p x^n. 
\end{equation}
All stated properties can be easily verified.
\end{proof}


\section{Discrete invariants}\label{sec:discreteinvariants}

In this section we shall describe all the discrete invariants for the classification of one dimensional superattracting germs up to (formal) conjugacy, over an algebraically closed field $\nK$ of characteristic $p>0$.

\begin{defi}\label{def:m}
Let $f(x)=\sum_{n} f_n x^n \in \nK[[x]]$ a (non-constant) formal power series. We set
$$
m(f):=\min\{\nu_p(n)\ |\ n \in \nN, f_n \neq 0\} \in \nN.
$$
\end{defi}

Notice that $m=m(f)$ if and only if there exists $g \in \nK[[y]]$ such that $g'(y) \not \equiv 0$ and $f=g \circ F^m$.

The number $m$ is an invariant of conjugacy. Indeed, $p^m$ is the inseparable degree of $\nK[[x]]/(f)$ over $\nK$.
It can also be shown by an easy computation.

\begin{prop}\label{prop:minvariant}
Let $f:(\nK,0) \to (\nK,0)$ be a superattracting germ. Let $m(f) \in \nN$ be defined as in Definition \ref{def:m}.
If $\wt{f}$ is formally conjugate to $f$, then $m(\wt{f})=m(f)$.
\end{prop}
\begin{proof}
Write
$$
f=g \circ F^m, \quad \wt{f}=\wt{g} \circ F^{\wt{m}},
$$
with $g', \wt{g}' \not \equiv 0$.

Let $\Phi \in \nK[[x]]$ be the conjugacy between $f$ and $\wt{f}$, and denote by $\Psi \in \nK[[x]]$ its inverse.
Hence $\wt{f}= \Phi \circ f \circ \Psi$.
It follows that
$$
\wt{f} = \Phi \circ g \circ F^m \circ \Psi = \Phi \circ g \circ T^m\Psi \circ F^m.
$$
Hence $\wt{m} \geq m$.
By switching the role of $f$ and $\wt{f}$, we get the opposite inequality $m \geq \wt{m}$, and hence $m=\wt{m}$ is an invariant of conjugacy.
\end{proof}

%
\begin{defi}\label{def:d}
Let $g(y)=\sum_{n} g_n y^n \in \nK[[y]]$ be a formal power series with $g' \not \equiv 0$. We set
$$
d(g)=\on{ord}_0(g):=\min\{n\ |\ g_n \neq 0\} \in \nN
$$
the order of vanishing of $g$ at $0$.
If $f=g \circ F^m$, we set $d(f):=d(g)$.
\end{defi}

Notice that $\on{ord}_0(f)=dp^m$, hence $d$ is an invariant of conjugacy.

%
\begin{defi}\label{def:r}
Let $g(y)=y^d \sum_n \eps_n y^n \in \nK[[y]]$ be a formal power series with $g' \not \equiv 0$ and $\eps_0 \neq 0$. We define recursively the sequence $r(g)=(r_u)_{u \in \nN}$ as following.
\begin{align*}
r_0&:= \min\{n \in \nN |\ \nu_p(d+n)=0 \text{ and } \eps_n \neq 0\},\\
r_u&:= \min\{n \in \nN |\ \nu_p(d+n)=u \text{ and } \eps_n \neq 0\} 	\wedge r_{u-1},\quad \text{for } u \in \nN^*.
\end{align*}
If $f=g \circ F^m$, we set $r(f):=r(g)$.
\end{defi}

Notice that $r_0+d-1 = \on{ord}_0(g')$, and $r_u$ is by definition a non-increasing sequence.
Moreover, $r_u=0$ for all $u \geq \nu_p(d)$.

We shall show that the sequence $r=r(f)$ is an invariant of conjugacy (see Lemma \ref{lem:BextRinvariant}).

\medskip

We conclude this section by studying how these discrete invariants behave under composition.
\begin{thm}\label{thm:comp}
Let $f', f'':(\nK,0) \to (\nK,0)$ be two superattracting germs over an algebraically closed field $\nK$ of characteristic $p>0$, and denote by $f=f''\circ f'$ their composition.
Set $m=m(f), d=d(f), r=r(f)$, $m'=m(f'), d'=d(f'), r'=r(f')$ and $m''=m(f''), d''=d(f''), r''=r(f'')$ given by Definitions \ref{def:m}, \ref{def:d}, \ref{def:r}.
Set $e=\nu_p(d)$, $e'=\nu_p(d')$, $e''=\nu_p(d'')$.
Then
\begin{align}
m &= m'+m'', \nonumber \\
d &= d'd'', \nonumber\\
e &= e'+e'', \nonumber\\
r_u &\geq \min\{d' r''_k + p^k r'_h\ |\ 0 \leq h \leq e', 0 \leq k \leq e'', h+k=u\} \quad \text{ for } 0 \leq u \leq e. \label{eqn:compr}
\end{align}
In the last relation, the equality holds in any of the following cases:
\begin{itemize}
\item the minimum is attained only for a unique choice of $(h,k)$ with $h+k=u$;
\item for $u=0, u=e$;
\item for a generic choice of $f', f''$.
\end{itemize}
\end{thm}
\begin{proof}
Write $f'=g' \circ F^{m'}$ and $f''=g'' \circ F^{m''}$ where the derivatives of $g'$ and $g''$ are not identically zero, and $F$ is the Frobenius automorphism.
Then
$$
f'' \circ f' = g'' \circ F^{m''} \circ g' \circ F^{m'} = g'' \circ T^{m''} g' \circ F^{m' + m''},
$$
where $T$ is the operator defined by \eqref{eqn:defT}.
By direct computation, the derivative of $g'' \circ T^{m''} g'$ is not identically $0$, and $m=m'+m''$.  

Now, since $p^m d=\on{ord}_0(f)$ is the order of vanishing of $f$ at $0$, and analogously for $f'$ and $f''$, we infer $p^m d = p^{m'}d' p^{m''}d''$, and hence $d=d' d''$.
The relation $e=e'+e''$ directly follows.

We now prove the the relation \eqref{eqn:compr}.
Notice that if $r''_k = r''_{k-1}$ for some $k\geq 1$, then $d'r''_k + p^k r'_h > d'r''_{k-1} + p^{k-1} r'_h \geq d'r''_{k-1} + p^{k-1} r'_{h+1}$ for any $h < e'$, and the minimum in \eqref{eqn:compr} is not attained by the values $(k,h)$.
Analogously if $r'_h=r'_{h-1}$ for some $h \geq 1$.

It follows that without loss of generality, we can suppose that $r'$ and $r''$ are strictly decreasing sequences, or equivalently that $\nu_p(r'_h)\wedge e'=h$, and $\nu_p(r''_k) \wedge e''=k$.
Set
$$
T^{m''}g'(y)=y^{d'} \overbrace{\sum_{n} a_n y^n}^{a(y)}, \qquad g''(z)=z^{d''} \sum_n b_n z^n.
$$
By the definition of $r'$ and $r''$, we have:
\begin{itemize}
\item $a_n = 0$ for any $n < r'_h$, $0 \leq h < e'$, $h=\nu_p(n)$, and analogously $b_n = 0$ for any $n < r''_k$, $0 \leq k < e''$, $k = \nu_p(n)$;
\item $a_{r'_h} \neq 0$ for any $0 \leq h \leq e'$, and analogously $b_{r''_k} \neq 0$ for any $0 \leq k \leq e''$.
\end{itemize}

We now want to study $g'' \circ T^{m''}g'(y)$ and its formal power series expansion. We get 
$$
g'' \circ T^{m''}g'(y) = \sum_j b_j y^{d'(d''+j)} \big(a(y)\big)^{d''+j}.
$$
For any $k=0, \ldots, e''$, set
$$
\eps^{(k)}(y):= \sumaa{j, \nu_p(j)\wedge e''=k}{0.3cm} b_j y^{d'j} \big(a(y)\big)^{d''+j}.
$$

The key of the proof of \eqref{eqn:compr} is given by the next lemma.
\begin{lem}\label{lem:compepsk}
Let $0 \leq k \leq e''$. Then we have
\begin{equation}\label{eqn:compepsk}
\eps^{(k)}(y) = b_{r''_k} y^{d'r''_k} \left( \frac{d''+r''_k}{p^k} \right) \sum_{h=0}^{e'} \left(a_{r'_h}^{p^k} y^{p^k r'_h} + o^{k+h}(y^{p^k r'_h})\right),
\end{equation}
where $o^{k+h}(y^{p^k r'_h})$ denotes a suitable formal power series on $y^{p^{k+h}}$ whose order or vanishing at $0$ (with respect to $y$) is $> p^k r'_h$. 
\end{lem}
\begin{proof}
We need to compute $(a(y))^{d''+j}$ for any $j$ with $\nu_p(j) \wedge e''=k$.
If $j < r''_k$, then $b_j = 0$.
Consider now $j = r''_k$. 
Notice that we also have $\nu_p(d''+r''_k) = k$.
In this case we have
\begin{equation}\label{eqn:compapower}
\big(a(y))^{d''+r''_k} = \big(T^k a(y^{p^k})\big)^{(d''+r''_k)/p^k}.
\end{equation}
Notice that $\nu_p\big(\frac{d''+r''_k}{p^k}\big)=0$.
The smallest degree that appear \eqref{eqn:compapower} of the form $y^{np^k}$ with $\nu_p(n)\wedge e'=h$ is given by
$$
\frac{d''+r''_k}{p^k} a_{r'_h}^{p^k} y^{p^k r'_h}.
$$
Hence we get an equation on the form \eqref{eqn:compepsk} when we consider $j\leq r''_k$ in the sum defining $\eps^{(k)}$.

We conclude by noticing that if $j > r''_k$, then $d'j > d'r''_k$, and the orders that appear for $j>r''_k$ are higher than the one got for $j=r''_k$. 
\end{proof}

We now conclude the proof of Theorem \ref{thm:comp}.

First, notice that $g'' \circ T^{m''}g'(y)=y^d \sum_{k=0}^{e''} \eps^{(k)}(y)$.
Hence, by Lemma \ref{lem:compepsk} we get
\begin{equation}\label{eqn:compconclusion}
r_u \geq \on{ord}_0\left(\sum_{k+h=u} \frac{d''+r''_k}{p^k} b_{r''_k}a_{r'_h}^{p^k} y^{d'r''_k+p^k r'_h}\right) \geq \min_{h+k=u} \{d'r''_k+p^k r'_h\},
\end{equation}
that gives us \eqref{eqn:compr}.
Notice that the coefficients in the sum of \eqref{eqn:compconclusion} are all different from zero.
For the properties of valuations, we have equalities in \eqref{eqn:compconclusion} when there is only one $(h,k)$ for which the minimum is attained.
In particular, this is always verified for $u=0$ (in this case $h=k=0$), or $u=e$ (in this case $h=e', k=e''$, end $r_e=0$).

More generally, suppose the minimum is attained for $(h,k) \in E$ for a suitable set of choices $E$.
Then we have equalities when
$$
\sum_{(h,k)\in E} b_{r''_k}a_{r'_h}^{p^k} \neq 0.
$$

\end{proof}

\begin{rmk}
In the notations of Theorem \ref{thm:comp}, suppose that $r'=r(f')$ is a strictly decreasing sequence.
Then we have that $r=r(f'' \circ f')$ is (generically) a strictly decreasing sequence.
Indeed, since $r'_h < r_{h-1}$, we have $d'r''_k + p^k r'_h < d'r''_k + p^k r'_{h-1}$.
It follows that for generic $f'$ and $f''$ we have
$$
r_u(f'' \circ f') = \min_{h+k=u} \{d'r''_k+p^k r'_h\} < \min_{h+k=u-1} \{d'r''_k+p^k r'_h\} = r_{u-1}(f'' \circ f')
$$
for any $u = 0, \ldots, e$.
\end{rmk}

Applying Theorem \ref{thm:comp} to the iterates of a superattracting germ, we get by induction the following corollary.
\begin{cor}
Let $f:(\nK,0) \to (\nK,0)$ be a superattracting germ over an algebraically closed field $\nK$ of characteristic $p>0$.
Let $m=m(f), d=d(f), r=r(f)$ be given by Definitions \ref{def:m}, \ref{def:d}, \ref{def:r}, and set $e=\nu_p(d)$.
Then
\begin{align*}
m(f^n) &= n m, \\
d(f^n) &= d^n, \\
e(f^n) &= n e, \\
r_0(f^n) &= 
\begin{cases}
r_0\frac{d^n-1}{d-1} & \text{if } e \geq 1,\\
0 & \text {if } e=0.
\end{cases}
\end{align*}
\end{cor}


\section{Normal forms}

In this section we describe the normal forms obtained, and state our main result.

We first need a few definitions and preparatory lemmas.

\begin{defi}\label{def:J}
Let $e \in \nN$, and let $r=(r_0, \ldots, r_e)$ (given by Definition \ref{def:r} as $r=r(f)$ for a suitable $f$) be a non-increasing sequence with $r_e=0$.
For any $n \in \nN$ and $0 \leq k \leq e$, we set
\begin{align*}
\mc{J}_k(n):=&
\begin{cases}
\frac{n-r_k}{p^k} & \text{if } k \leq \nu_p(n) \text{ and } n > r_k,\\
0 & \text{otherwise};
\end{cases}\\
\mc{J}(n):=&\max\{\mc{J}_k(n), 0 \leq k \leq e\}.
\end{align*}
\end{defi}

\begin{rmk}
Notice that if $r_k<r_{k-1}$ (for $k < e$), by construction $\nu_p(r_k) = k$, and $\mc{J}_k(n) \in \nN$ for any $n$.
If $r_k=r_{k-1}$, then we could have $\mc{J}_k(n) \in \nQ \setminus \nN$. Nevertheless, in this case we have
$$
\mc{J}_k(n)= \frac{n-r_k}{p^k} = \frac{n-r_{k-1}}{p^k} < \frac{n-r_{k-1}}{p^{k-1}} = \mc{J}_{k-1}(n).
$$
It follows that $\mc{J}(n)$ is always an integer.
\end{rmk}

\begin{lem}\label{lem:Jnonempty}
The set $\{n\ |\ \mc{J}(n)=j\}$ is non-empty for any $j \in \nN$.
\end{lem}
\begin{proof}
It is straightforward to show that
$$
N'(j) := \min_k\{r_k + p^k j\}
$$
satisfies $\mc{J}(N'(j))=j$.
\end{proof}

\begin{lem}\label{lem:howJis}
The map $\mc{J}$ defined in Definition \ref{def:J} satisfies the following properties.
\begin{enumerate}[(i)]
\item $\{n\ |\ \mc{J}(n)=0\} = \{n \in \nN\ |\ \nu_p(n)=u < e, n \leq r_u\} \cup \{0\}$.
\item If $\mc{J}(n) < r_0/(p-1)$, then $n < p r_0/(p-1)$.
\item For any $j \geq r_0/(p-1)$, we have $\{n\ |\ \mc{J}(n)=j\}= \{r_0+j\}$.
\end{enumerate}
\end{lem}
\begin{proof}
The first property is straightforward.

Suppose $n \geq pr_0/(p-1)$, or equivalently $n/p \leq n-r_0$.
Then for any $k\geq 1$ we have
$$
\mc{J}_k(n)=\frac{n-r_k}{p^k} \leq \frac{n}{p} \leq n-r_0 = \mc{J}_0(n).
$$
Hence $\mc{J}(n)=\mc{J}_0(n) = n-r_0 \geq r_0/(p-1)$.

Suppose now $n< p r_0 /(p-1)$.
Then for any $k \geq 1$ we have
$$
\mc{J}_k(n)=\frac{n-r_k}{p^k} \leq \frac{n}{p} < \frac{r_0}{p-1}.
$$
The statement follows.
\end{proof}
\begin{rmk}\label{rmk:howJis}
Lemma \ref{lem:howJis} can be easily improved. Indeed one can show that the property stated in (iii) holds for any $j$ so that
$$
j \geq \max_{k \geq 1} \frac{r_0 - r_k}{p^k-1}.
$$
\end{rmk}

We are now able to state the classification result.

\begin{thm}\label{thm:Bext}
Let $\nK$ be an algebraically closed field of characteristic $p > 0$.
Let $f:(\nK,0) \to (\nK,0)$ be a superattracting germ.
Set $m=m(f), d=d(f), r=r(f)$ as in Definitions \ref{def:m}, \ref{def:d}, \ref{def:r}.
Set $e=\nu_p(d)$, and denote by $\mc{J}$ the map introduced in Definition \ref{def:J}.

For any $0 < j < r_0/(p-1)$, pick $N(j)$ such that $\mc{J}(N(j))=j$.

Then $f$ is conjugate to a germ of the form:
\begin{equation}\label{eqn:Bextnormalform}
\wt{f}(x)=\big(x^{p^m}\big)^d a\big(x^{p^{m}}\big),
\end{equation}
where either $e=0$ and $a \equiv 1$, or $a \in \nK[y]$ is a polynomial of degree $< pr_0/(p-1)$.
In the latter case, write $a$ under the form
$$
a(y)=\sumaa{0 \leq n < p r_0/(p-1)}{0.6cm} a_n y^n.
$$
Then $a$ also satisfies the following conditions.
\begin{enumerate}[(i)]
\item We have $a_{r_e}=a_0=1$.
\item For any $0 \leq u < e$, $n < r_u$ with $\nu_p(n) = u$, then $a_n=0$.
\item For any $0 \leq u < e$, then $a_{r_u} \neq 0$.
\item For any $0 < j < r_0/(p-1)$, $a_{N(j)}=0$.  
\end{enumerate}

Finally, there exist only finitely many germs of the form \eqref{eqn:Bextnormalform} satisfying all conditions (i--iv) that are conjugate to $f$.
\end{thm}

\begin{rmk}\label{rmk:normalformse0}
Let $f:(\nK,0)\to(\nK,0)$ be a superattracting germ written as in the statement of Theorem \ref{thm:Bext}.
When $\nu_p(d)=0$, we get the classical B\"ottcher normal forms $\wt{f}(x)=x^{p^m d}$ (see \cite{sternberg:localcontractionstheorempoincare,rosay-rudin:holomorphicmaps,berteloot:methodeschangementechelles}).
This case was already known by experts, see Theorem \ref{thm:Beasy} for a direct proof.
\end{rmk}

\begin{rmk}\label{rmk:normalformse1N'}
The normal forms provided by Theorem \ref{thm:Bext} depend on the choice of $N(j)$ for $0 < j < r_0/(p-1)$.
Suppose we pick $N(j)=N'(j)$ defined in Lemma \ref{lem:Jnonempty}.
In the case when $\nu_p(d)=1$, we get normal forms \eqref{eqn:Bextnormalform} with
$$
a(y)=1+ \sumaa{\substack{r \leq n < p r_0/(p-1) \\ \nu_p(n)=0}}{0.4 cm} a_n y^n, \qquad a_r \neq 0.
$$
When $m=0$ and $d=p$, these normal forms are the one proposed in \cite{gardnerspencer:phdthesis}.
\end{rmk}

To get Theorem \ref{thm:Bhard}, we need to consider another choice for $N(j)$, defined by used a non-standard total order on $\nN$.

\begin{defi}\label{def:preceq}
We denote by $\preceq$ the total order on $\nN$ given by the lexicographic order on $(\nu_p(n) \wedge e, n)$.
\end{defi}

\begin{es}
Suppose $p=3$ and $e=2$.
Then the order $\preceq$ is given by
\begin{align*}
\nu_p(n)=0:\qquad &1 \prec 2 \prec 4 \prec 5 \prec 7 \prec 8 \prec \cdots \\
\nu_p(n)=1:\qquad &\prec 3 \prec 6 \prec 12 \prec 15 \prec 21 \prec 24 \prec \cdots \\
\nu_p(n)\geq 2:\qquad &\prec 0 \prec 9 \prec 18 \prec 27 \prec 36 \prec 45 \prec 54 \prec 63 \prec \cdots \\
\end{align*}
\end{es}

\begin{rmk}
Let $\mc{J}$ be given by Definition \ref{def:J}, and set
$$
N''(j):={\textstyle\min_\preceq}\{n\ |\ \mc{J}(n)=j\}.
$$
Notice that if $\nu_p(j+r_0)=0$, then $N''(j)=j+r_0$.

By picking $N(j)=N''(j)$ in Theorem \ref{thm:Bext}, we get a normal form \eqref{eqn:Bextnormalform} with $a(y)=\sum_n a_n y^n$ satisfying the condition $a_n=0$ for any $n \neq r_0, \nu_p(n)=0$.
In particular, we get Theorem \ref{thm:Bhard}.

In the special case when $\nu_p(d)=1$, we get normal forms \eqref{eqn:Bextnormalform} with
$$
a(y)=a_{r_0} y^{r_0} + \sumaa{\substack{0 \leq s < r_0/(p-1) \\ \nu_p(s+r_0)=0}}{0.4 cm} a_{ps} y^{ps}, \qquad a_0=1, a_{r_0} \neq 0.
$$
\end{rmk}

\begin{es}
As an example of what one can get in general, let us consider the following situation: $p=3, \nu_p(d)=2, r=(19, 12, 0)$.
The following table summarizes the values of $\mc{J}(n)$ for $n <30$.
$$
\small
\begin{array}{c|*{30}{p{0.1cm}}}
n 		& 0& 1& 2& 3& 4& 5& 6& 7& 8& 9&10&11&12&13&14&15&16&17&18&19&20&21&22&23&24&25&26&27&28&29\\
\hline
\mc{J}_0(n)	&  &  &  &  &  &  &  &  &  &  &  &  &  &  &  &  &  &  &  & $\times$ & 1& 2& 3& 4& 5& 6& 7& 8& 9&10\\
\mc{J}_1(n)	&  &  &  &  &  &  &  &  &  &  &  &  & $\times$ &  &  & 1&  &  & 2&  &  & 3&  &  & 4&  &  & 5&  &  \\
\mc{J}_2(n)	& $\times$ &  &  &  &  &  &  &  &  & 1&  &  &  &  &  &  &  &  & 2&  &  &  &  &  &  &  &  & 3&  &  \\
\hline
\mc{J}(n)	& $\times$ &  &  &  &  &  &  &  &  & 1&  &  & $\times$ &  &  & 1&  &  & 2& $\times$ & 1& 3& 3& 4& 5& 6& 7& 8& 9&10\\
\end{array}
$$
The $\times$ are associated to numbers $n$ so that $r_u=n$ for some $u$.
Here we get
$$
\{n\ |\ \mc{J}(n)=j\}
=
\begin{cases}
\{9,15,20\} & \text{if }j=1,\\
\{18\} & \text{if }j=2,\\
\{21, 22\} & \text{if }j=3,\\
\{j+19\} & \text{if }j\geq 4.
\end{cases}
$$
The case $j \geq 4$ follows directly by Remark \ref{rmk:howJis}, since
$$
\max\left\{\frac{r_0-r_1}{p-1}, \frac{r_0-r_2}{p^2-1}\right\} = \max\{3.5, 2.375\} = 3.5.
$$
Here by taking for example $m=0$, $d=3^2 \cdot 2=18$, $N(1)=15$ and $N(3)=21$, for $m=0$ we get normal forms
$$
\wt{f}(x)=x^{18}(1+a_9 x^9 + a_{12} x^{12} +a_{19} x^{19} + a_{20} x^{20} + a_{22} x^{22}).
$$
\end{es}

We conclude this section by noticing that (unlike the case of characteristic zero) not all normal forms can be obtained as the action at infinity of polynomial mappings.

\begin{cor}\label{cor:polynomials}
Let $P \in \nK[z]$ be a polynomial of degree $\geq 2$. Denote by $f$ the superattracting germ obtained by considering the action of $P$ at $\infty \in \nP^1_\nK$.
Set $d=d(f)$ and $r=r(f)$ given as in Definitions \ref{def:d} and \ref{def:r}.
Then $r_0\leq d$.
\end{cor}
\begin{proof}
Write $P=Q \circ F^m$, where $Q \in \nK[w]$ is such that $Q' \not \equiv 0$, and $F$ is the Frobenius automorphism.
Write $Q(w)=w^d - b_1 w^{d-1} - \ldots - b_d$.
Then in the local coordinates $x=1/z$, the germ $f$ is equal to
\begin{equation}\label{eqn:germatinfty}
P(x)=x^{p^m d}\left(1-\sum_{n=1}^d b_n x^{np^m}\right)^{-1}.
\end{equation}
From this formula, it can be easily verified that
$$
r_0 = \min\{1 \leq n \leq d\ |\ b_n \neq 0, \nu_p(n)=0\}.
$$
In particular $r_0 \leq d$.
\end{proof}

In view of Corollary \ref{cor:polynomials}, one can ask the following question.

\begin{quest}
Let $f:(\nK,0) \to (\nK,0)$ be a superattracting germ in normal form \eqref{eqn:Bextnormalform}.
Set $d=d(f)$ and $r=r(f)$ given as in Definitions \ref{def:d} and \ref{def:r}.
Suppose $r_0< d$.
Can $f$ be obtained as the action at infinity of a polynomial mapping $P \in \nK[z]$?
\end{quest}

It can be easily shown that the answer to this question is positive at least when $\deg a \leq d$ in \eqref{eqn:Bextnormalform}, so in particular when $r_0 \leq d(1-1/p)$.

\section{Analytic normal forms when $d$ and $p$ are coprime}

In this section, we prove Theorem \ref{thm:Bext} and its analytic counterpart when $\nu_p(d)=0$.

\begin{thm}\label{thm:Beasy}
Let $\nK$ be a complete normed algebraic closed field of characteristic $p > 0$.
Let $f:(\nK,0) \to (\nK,0)$ be a superattracting germ.
Set $m=m(f)$ and $d=d(f)$ as in Definitions \ref{def:m} and \ref{def:d}, and assume that $d$ and $p$ are coprime.

Then $f$ is analytically conjugate to the germ
\begin{equation}\label{eqn:Beasynormalform}
x \mapsto \big(x^{p^m}\big)^d.
\end{equation}
\end{thm}
\begin{proof}
Set $y=x^{p^m}$, and
\begin{equation}\label{eqn:Beasystart}
f(x)=g(y)=C y^d (1+\eps(y)),
\end{equation}
with $C \neq 0$ and $\eps:(\nK,0) \to (\nK,0)$ an analytic germ with $\eps(0) = 0$.

We want to find a conjugacy between $f$ and $\wt{f}(x) = C (x^{p^m})^d = C y^d$.

Up to linear conjugacy, we can suppose:
\begin{itemize}
\item $\abs{C} < 1$, and there exists $0<\alpha<1$ such that $\abs{g(y)} \leq \alpha \abs{y}$ for $\abs{y} \le 1$;
\item $\abs{\frac{d\eps}{dy}(0)}<1$, and there exists $0 < \beta < 1$ such that $\abs{\eps(y)} \leq \beta \abs{y}$ for $\abs{y} \le 1$.
\end{itemize}

Let us consider a local diffeomorphism of the form
$$
\Phi(x)=x\phi(x),
$$
with $\phi(0)=1$.

Considering the conjugacy relation $\Phi \circ f = \wt{f} \circ \Phi$, we get
\begin{align*}
\Phi \circ f(x)&= C y^d\big(1+\eps(y)\big)\phi \circ g(y)\\
\wt{f} \circ \Phi(x)&= C y^d \big(\phi(x)\big)^{dp^m}= C y^d (T^m\phi(y))^d,
\end{align*}
where $T$ is the operator defined by \eqref{eqn:defT}.
In particular we have to solve
\begin{equation}\label{eqn:Beasytosolve}
\big(1+\eps(y)\big)\phi \circ g(y)
= (T^m\phi(y))^d.
\end{equation}
A solution to this equation is given by the formal product
\begin{equation}\label{eqn:Beasysolution}
\phi(y)=\prod_{n=0}^\infty \big(1+\eps^{(n)}(y)\big)^{d^{-n-1}},
\end{equation}
where $\eps^{(n)}:(\nK,0) \to (\nK,0)$ are analytic germs satisfying the relations
\begin{equation}\label{eqn:defeps}
T^m\eps^{(0)}=\eps, \qquad T^m \eps^{(n+1)} = \eps^{(n)} \circ g \mbox{ for } n \geq 0.
\end{equation}
Notice that the single factor $\big(1+\eps^{(n)}(y)\big)^{d^{-n-1}}$ is well defined, since the equations in \eqref{eqn:defeps} have always solutions, and $d$ and $p$ are coprime (see Remark \ref{rmk:rootwelldefined}).

Let us now show that the formal product converges, thus defining an analytic change of coordinates.

Proceeding by induction on $n$, when $m=0$ we infer
$$
\abs{\eps^{(n)}(y)} \leq \beta \alpha^n \abs{y}
$$
when $\abs{y} \ll 1$.
When $m > 0$, we get
$$
\abs{\eps^{(n)}(y)} \leq M \gamma^n \abs{y},
$$
where
$$
M = \left(\alpha^{(1-p^{-m})^{-1}} \beta\right)^{p^{-m}}, \qquad \gamma = \beta^{p^{-m}}<1.
$$
In both cases, $d^{-n} \abs{\eps^{(n)}(y)} \to 0$ when $n \to +\infty$ for $\abs{y} \ll 1$.
By Proposition \ref{prop:prodconverges}, the product in \eqref{eqn:Beasysolution} converges for $\abs{y}$ small enough.

Hence $f$ and $\wt{f}:x \mapsto Cx^{p^m d}$ are analytically conjugate.
Up to a linear change of coordinates, we can now get $C=1$.
\end{proof}


\section{Proof of Theorem \ref{thm:Bext}}

This section is completely devoted to the proof of Theorem \ref{thm:Bext}.

Write $f$ under the form $g \circ F^m$, where
$$
g(y)=y^d\big(1+\eps(y)\big), \qquad 1+\eps(y)=\sum_{n=0}^\infty \eps_n y^n,
$$
with
\begin{enumerate}[(i)]
\item $\eps_0 = 1$;
\item for any $0 \leq u < e$, $n < r_u$ with $\nu_p(n) = u$, then $\eps_n=0$;
\item for any $0 \leq u < e$, then $\eps_{r_u} \neq 0$.
\end{enumerate}
Condition (i) can be achieved up to a linear change of coordinates, while conditions (ii--iii) follow directly from the Definition \ref{def:r} of $r=r(f)$.

Let $\wt{f}$ be another superattracting germ of the form $\wt{f}=\wt{g} \circ F^m$, with
$$
\wt{g}(y)=y^d\big(1+\wt{\eps}(y)\big), \qquad 1+\wt{\eps}(y)=\sum_{n=0}^\infty \wt{\eps}_n y^n,\quad \wt{\eps}_0=1.
$$
Consider a change of coordinates $\Phi\in \nK[[x]]$ of the form
$$
\Phi(x)=x\phi(x), \qquad \phi(x)= \sum_{n=0}^\infty \phi_n x^n, \quad \phi_0=1.
$$

Considering the conjugacy relation $\Phi \circ f = \wt{f} \circ \Phi$, we get
\begin{align*}
\Phi \circ f(x) &= y^d \big(1+\eps(y)\big) \phi\Big(y^d \big(1+\eps(y)\big)\Big),\\
\wt{f} \circ \Phi(x) &= \big(x\phi(x)\big)^{dp^m}\Big(1+ \wt{\eps}\big(x\phi(x)\big)\Big) = y^d \big(T^m\phi(y)\big)^d\Big(1+ \wt{\eps}\big(y T^m\phi(y)\big)\Big),
\end{align*}
where the operator $T$ is defined by \eqref{eqn:defT}.
In particular we have to solve
\begin{equation}\label{eqn:Btosolve}
\big(1+\eps(y)\big) \phi\Big(y^d \big(1+\eps(y)\big)\Big)
=\big(T^m\phi(y)\big)^d \Big(1+\wt{\eps}\big(y T^m\phi(y)\big)\Big).
\end{equation}
We recall that the unknowns of this equation are $\phi$ and $\wt{\eps}$, while $\eps$ is the datum.
We now develop both sides of \eqref{eqn:Btosolve} in formal power series.

Denote by $\lhs(y)=\sum_n \lhs_n y^n$ and $\rhs(y)=\sum_n \rhs_n y^n$ the left hand side and right hand side of \eqref{eqn:Btosolve}.
We first need a few elementary lemmas that will help the needed computations.

\begin{lem}\label{lem:easy1}
Let $\psi(y) = \sum_{n=0}^\infty \psi_n y^n \in \nK[[y]]$ be a formal power series, and $h \in \nN$.
Then
$$
\big(\psi (y)\big)^h=\sum_{J \in \nN^h} \psi_J y^{\abs{J}} = \sum_{n=0}^\infty \Bigg(\sumab{J \in \nN^h,\ \abs{J}=n}{0.0cm}{0.4cm} \psi_J\Bigg) y^n,
$$
where if $J=(j_1 \ldots, j_h)$, we set $\psi_J=\psi_{j_1} \cdots \psi_{j_h}$ and $\abs{J}=j_1 + \cdots + j_h$.
\end{lem}

\begin{lem}\label{lem:easy2}
Let $\psi(y) = \sum_{n=0}^\infty \psi_n y^n \in \nK[[y]]$ be a formal power series, $h, n \in \nN$ such that $\nu_p(h)>\nu_p(n)$.
Then
\begin{equation}\label{eqn:easyimp}
\sumaa{\substack{J \in \nN^h\\\abs{J}=n}}{0.1cm} \psi_J = 0.
\end{equation}
\end{lem}
\begin{proof}[Proof of Lemmas \ref{lem:easy1} and \ref{lem:easy2}]
The proof of the first lemma is trivial.
To prove the second lemma, it suffices to notice that the sum in \eqref{eqn:easyimp} gives the term of degree $n$ of $\psi^h(y)$.
Set $k=\nu_p(h)$. Then $\psi^h(y)= (T^k\psi(y^{p^k}))^{h/p^k}$ depends only on $y^{p^k}$, hence any term of degree $n$ with $\nu_p(n)<k$ is zero.
\end{proof}

We can now come back to the proof of Theorem \ref{thm:Bext}.
By expressing the left and right hand sides of \eqref{eqn:Btosolve} in formal power series, and using Lemma \ref{lem:easy1} and Proposition \ref{prop:defT}, we get
\begin{align*}
\lhs(y)
&= \big(1+\eps(y)\big) \phi\big(y^d (1+\eps(y))\big) 
= \sum_j \phi_j y^{dj} \big(1+\eps(y)\big)^{j+1}
= \sum_j \phi_j y^{dj} \sum_{J \in \nN^{j+1}} \eps_J y^{\abs{J}},\\
\rhs(y)
&= \big(T^m\phi(y)\big)^d\Big(1+ \wt{\eps}\big(y T^m\phi(y)\big)\Big) 
= \sum_i \wt{\eps}_i y^i \big(T^m\phi(y)\big)^{d+i}
= \sum_j \wt{\eps}_i y^i \sum_{I \in \nN^{d+i}} \phi_I^{p^m} y^{\abs{I}}.\\
\end{align*}
Again by Lemma \ref{lem:easy1} we infer
\begin{equation} \label{eqn:Bcoeff}
\lhs_n = \sumaa{\substack{j \geq 0, J \in \nN^{j+1}\\ dj + \abs{J}=n}}{0.5cm} \phi_j \eps_J,
\qquad\qquad
\rhs_n = \sumaa{\substack{i \geq 0, I \in \nN^{d+i}\\ i + \abs{I}=n}}{0.5cm} \wt{\eps}_i \phi_I^{p^m}.
\end{equation}

To analyze the combinatorics of the equations $\lhs_n=\rhs_n$ we need a few preliminary lemmas.

\begin{lem}\label{lem:BextRinvariant}
The equations $\rhs_n=\lhs_n$ for $\mc{J}(n)=0$ admit a unique solution $\wt{\eps}_n =\eps_n$.
In particular, we infer $\wt{\eps}_{r_u}= \eps_{r_u}$ for $0 \leq u \leq e$, $\wt{\eps}_n=0$ otherwise, and the sequence $r=(r_u)_u$ introduced in Definition \ref{def:r} is an invariant of conjugacy.
\end{lem}
\begin{proof}
We proceed by induction on $n$ such that $\mc{J}(n)=0$, with respect to the order $\preceq$.
Recall that by Lemma \ref{lem:howJis}, we have
$$
\{n\ |\ \mc{J}(n)=0\} = \{n \in \nN\ |\ \nu_p(n)=u < e, n \leq r_u\} \cup \{0\}.
$$
For $n=0$, the statement is trivial.

Suppose we proved the statement for any $n' \prec n$ with $\mc{J}(n')=0$.
Consider the equation $\lhs_n=\rhs_n$.

We first show $\lhs_n=\eps_n$.
In the sum defining $\lhs_{n}$ in \eqref{eqn:Bcoeff}, we have the condition $dj+\abs{J}=n$.
In particular $\abs{J} \leq n < r_{u'}$ for any $u'<u$.

Write $J=(J_0, \ldots, J_j)$.
If there exists $h$ such that $\nu_p(J_h)< u$, then by induction hypothesis $\eps_{J_h}=0$, hence $\eps_J=0$.

Suppose now $\nu_p(J_h) \geq u$ for any $h$.
Since $\nu_p(d)=e > u=\nu_p(n)$, we infer $\nu_p(n-dj)=u$.
It follows that there exists $h$ such that $\nu_p(J_h)=u$. If $J_h < n$, by induction hypothesis $\eps_{J_h}=0$, hence $\eps_J=0$.
If $J_h=n$, then $j=0$, and we get $\lhs_n=\phi_0\eps_n=\eps_n$.

We now show that $\rhs_n=\wt{\eps}_n$, and conclude the proof.

Suppose $\nu_p(i)>u$.
Since $\nu_p(n)=u$, by Lemma \ref{lem:easy2} we infer $\displaystyle \sumaa{I \in \nN^{d+i}, \abs{I}=n-i}{0.6cm} \phi_I^{p^m} = 0$.

Suppose $\nu_p(i)\leq u$.
We can suppose $i \preceq n$.
If $i \prec n$, by induction hypothesis we get $\wt{\eps}_i=0$.
If $i = n$, then $\abs{I}=0$, and we get $\rhs_n=\wt{\eps}_n \phi_0^{p^m d}= \wt{\eps}_n$.
\end{proof}

\begin{lem}\label{lem:Bextrhslot}
For any $n \in \nN$, we have
$$
\rhs_n = \wt{\eps}_n + \sum_{k=0}^{e \wedge u} \left(\frac{d+r_k}{p^k}\right) \wt{\eps}_{r_k} \phi_{(n-r_k)/p^k}^{p^{m+k}} +
\sum_{k=0}^{e \wedge u} Q_k\Big(\phi_h, \wt{\eps}_i\ \big|\ h<(n-r_k)/p^k; \nu_p(i) \wedge e = k, i<n\Big),
$$
where $u=\nu_p(n)$, and $Q_k$ denotes a suitable polynomial that depends on $\phi$ and $\wt{\eps}$ as indicated.
Here we set $\phi_j=0$ whenever $j \not \in \nN$. 
\end{lem}
\begin{proof}
Set $k=\nu_p(i)$.
For any $k<e$ and $i < r_k$, by Lemma \ref{lem:BextRinvariant} we get $\wt{\eps}_i=0$.

If $u < e$ and $k > u$, then by Lemma \ref{lem:easy2} we get
$$
\sumaa{\substack{J \in \nN^{d+i}\\ \abs{J}=n-i}}{0.1cm} \phi_J^{p^m} = 0,
$$
and $\rhs_n$ does not depend on $\wt{\eps}_i$.

Suppose $k \leq u \wedge e$.
Notice that
$$
\sumaa{\substack{J \in \nN^{d+i}\\ \abs{J}=n-i}}{0.1cm} \phi_J^{p^m} = 
\sumaa{\substack{H \in \nN^{(d+i)/p^k}\\ \abs{H}=(n-i)/p^k}}{0.3cm} \phi_H^{p^{m+k}}.
$$
It follows that the highest order term that can appear in $\rhs_n$ depending on $\wt{\eps}_i$ with $\nu_p(i)=k \leq u \wedge e$ is obtained when $i=r_k$, and given by
\begin{equation}\label{eqn:Bextrhscoeff}
\left(\frac{d+r_k}{p^k}\right)\wt{\eps}_{r_k} \phi_{(n-r_k)/p^k}^{p^{m+k}}.
\end{equation}
Notice that $\frac{d+r_k}{p^k}\wt{\eps}_{r_k} \neq 0$ for any $0 \leq k \leq u \wedge e$, and $d/p^k=0$ as an element of $\nK$.

If $u<e$ we are done.
If $u \geq e$, then $r_u=0$. It follows that the highest order term that appear in $\rhs_n$ depending on $\wt{\eps}_i$ for $\wt{\eps}_i \geq e$ is still given by \eqref{eqn:Bextrhscoeff} with $k=e$, and we are done.
\end{proof}

\begin{lem}\label{lem:Bextlhslot}
For any $n \in \nN$, we have
$$
\lhs_n = \kronecker{d}{p^e} \phi_{n/p^e} + P\big(\phi_j\ |\ j < (n-r_u)/p^u,\ u = \nu_p(n) \wedge e \big),
$$
where $\kronecker{}{}$ denotes the Kronecker delta function, and $P$ denotes a suitable polynomial that depends on $\phi$ as indicated.
Again, we set $\phi_j=0$ whenever $j \not \in \nN$. 
\end{lem}
\begin{proof}
Consider the sum defining $\lhs_n$ in \eqref{eqn:Bcoeff}.
The indices $j$ and $J$ have to satisfy $dj + \abs{J} = n$, hence
$$
j \leq \frac{n - \abs{J}}{d}.
$$
Set $u=\nu_p(n)$, and suppose $u < e$.
Then $\nu_p(\abs{J})=\nu_p(n-dj)=u$.
Write $J=(J_0, \ldots, J_j)$.
Then from $\nu_p(\abs{J})=u$ we infer that there exists $h$ such that $\nu_p(J_h)\leq u$.
It follows that either $\eps_J=0$, or $\abs{J} \geq r_u$, and hence
$$
j \leq \frac{n - r_u}{d} < \frac{n-r_u}{p^u},
$$
since $d > p^u$.
Suppose now $u \geq e$. We have
$$
j \leq \frac{n}{d} \leq \frac{n-r_e}{p^e}.
$$
The last inequality is strict unless $d=p^e$.
In this case the only non-zero term with $j=n/p^e$ arises when $J=(0, \ldots, 0)$, and is given by $\eps_0^{n/p^e+1} \phi_{n/p^e}=\phi_{n/p^e}$. 
\end{proof}

We are now ready to prove Theorem \ref{thm:Bext}, by solving the equation $\lhs_n=\rhs_n$ for any $n$.
We recall that the unknowns are given by $\phi_j$ and $\wt{\eps}_i$.

We proceed by recursion on $j=\mc{J}(n)$, as follows.

\medskip

Lemma \ref{lem:BextRinvariant} provides the basis of the induction:
for any $n$ with $\mc{J}(n)=0$, we solve $\lhs_n=\rhs_n$ and infer $\wt{\eps}_n=\eps_n$.
\smallskip

Suppose now that $\phi_h$ and $\wt{\eps}_i$ are defined for any $h<j$ and $\mc{J}(i)<j$.
We shall solve the equation $\lhs_n=\rhs_n$ for all $n$ such that $\mc{J}(n)=j$, fixing the value of $\wt{\eps}_n$ for all such $n$ and of $\phi_j$.

We now claim that the equation $\lhs_n=\rhs_n$ can be written under the form
\begin{equation}\label{eqn:Bexteqnlot}
\wt{\eps}_n + \sumaa{\substack{0 \leq k \leq e \\ J_k(n)=j}}{0.3 cm}\Big(\frac{d + r_k}{p^k} \wt{\eps}_{r_k} \phi_j^{p^{m+k}}\Big)
- \kronecker{d}{p^e} \kronecker{n}{jd} \phi_j
= Q\Big(\phi_h, \wt{\eps}_i\ \big|\ h<j; i< N(j)\Big),
\end{equation}
where $Q$ is a suitable polynomial depending on $\phi_h$ and $\wt{\eps}_i$ as indicated.

Indeed, by Lemma \ref{lem:Bextrhslot} we have
$$
\rhs_n = \wt{\eps}_n + \sum_{k=0}^{e \wedge \nu_p(n)} \frac{d+r_k}{p^k} \wt{\eps}_{r_k} \phi_{\mc{J}_k(n)}^{p^{m+k}} +
\sum_{k=0}^{e \wedge \nu_p(n)} Q_k\Big(\phi_h, \wt{\eps}_i\ \big|\ h<\mc{J}_k(n); \nu_p(i) \wedge e = k, i<n\Big).
$$
Notice that if $\mc{J}_k(n) < \mc{J}(n) = j$, then $\frac{d+r_k}{p^k} \wt{\eps}_{r_k} \phi_{\mc{J}_k(n)}^{p^{m+k}}$ depends on $\phi_h$ with $h<j$.
We now show that if $i$ is such that $\nu_p(i) \wedge e = k$ and $i < n$, then $\mc{J}(i) < \mc{J}(n)$.
Notice that $\nu_p(i) \wedge e \leq \nu_p(n) \wedge e$. It follows that
$$
\mc{J}(i)\leq \max_{k \leq \nu_p(n) \wedge e} \frac{i-r_k}{p^k} < \max_{k \leq \nu_p(n) \wedge e} \frac{n-r_k}{p^k} = \mc{J}(n).
$$

By Lemma \ref{lem:Bextlhslot}
$$
\lhs_n =\kronecker{d}{p^e} \phi_{n/p^e} + P\big(\phi_h\ |\ h < \mc{J}_{u \wedge e}(n) \leq \mc{J}(n)=j\big);
$$
hence \eqref{eqn:Bexteqnlot} holds.

We now come back to the resolution of $\lhs_n = \rhs_n$ for any $n$ with $\mc{J}(n)=j$.
The value of $Q$ in \eqref{eqn:Bexteqnlot} is determined by induction hypothesis (for any such $n$).

For $n = N(j)$ we solve the equation $\lhs_{N(j)}=\rhs_{N(j)}$ as follows.
Set $\wt{\eps}_{N(j)}=0$.
The left hand side of \eqref{eqn:Bexteqnlot} is a polynomial $R(\phi_j)$, of the form
$$
R(z)=\sumaa{\substack{0 \leq k \leq e \\ J_k(n)=j}}{0.3 cm} R_k z^{p^{m+k}} - \kronecker{d}{p^e} \kronecker{n}{jd} z,
$$
with $R_k \neq 0$ for any $k$.
We need to check that if $\{k\ |\ J_k(n)=j\}=\{0\}$, then $\kronecker{d}{p^e} \kronecker{n}{jd}=0$.
Suppose by contradiction that $d=p^e$ and $n=jp^e$. Then $J_e(n)=j$.
If $e > 0$ we have a contradiction.
If $e=0$, then $d=p^0=1$, in contradiction with the fact that $f$ is superattracting and $dp^m \geq 2$.
Hence the polynomial $R$ is non-null, and we can find $\phi_j \in \nK$ that solves the equation $\lhs_{N(j)}=\rhs_{N(j)}$.

For all $n \neq N(j)$ with $\mc{J}(n)=j$, from equation \eqref{eqn:Bexteqnlot} we infer that there exists a (unique) $\wt{\eps}_n$ solving the equation.

With this procedure, by Lemma \ref{lem:howJis}, we get that $\wt{\eps}_n=0$ for:
\begin{itemize}
\item $n = N(j)$ for all $0 < j < r_0/(p-1)$;
\item $n < r_u$, $\nu_p(n)=u$, $0 \leq u < e$;
\item $n \geq p r_0/(p-1)$.
\end{itemize}
Hence the map $\wt{f}(x)=y^d \wt{\eps}(y)$ provides a normal form \eqref{eqn:Bextnormalform} satisfying all conditions (i--iv), and we are done.

Notice that the values of $\wt{\eps}_n$ for any $n$ depend on the value of $\phi_j$ only for $j < r_0/(p-1)$.
These coefficients are uniquely determined up to a finite number of choices.
The last assertion of the statement follows.

\begin{rmk}\label{rmk:Bextgeneralization}
Theorem \ref{thm:Bext}, or similar results, hold over fields $\nK$ with milder hypotheses than being algebraically closed.

Let now $\nK$ be any field of characteristic $p>0$.
It can be shown that if $f:(\nK,0) \to (\nK,0)$ is a superattracting germ with $f' \not \equiv 0$ (or analogously $m(f)=0$), then $f$ admits a polynomial normal form of degree $\leq d+ r_0 p/(p-1)$.
More generally, if $m(f)>0$, and $\nK$ is closed by taking $p^m$-th roots, $f$ admits again a polynomial normal form of degree $\leq p^m(d+r_0p/(p-1))$.

Indeed, by Lemma \ref{lem:howJis}.(iii) when $n$ is \emph{strictly} bigger than $r_0p/(p-1)$ we have that $\mc{J}_0(n) > \mc{J}_k(n)$ for any $k > 0$.
Hence by \eqref{eqn:Bexteqnlot} the equation $\lhs_n=\rhs_n$ is of the form
$$
\wt{\eps}_n + r_0 \wt{\eps}_{r_0} \phi_{n-r_0}^{p^{m}} = Q\Big(\phi_h, \wt{\eps}_i\ \big|\ h<n-r_0; i< n\Big).
$$
This equation can be solved setting $\wt{\eps}_n=0$ as far as we can take the $p^m$-th root of elements in $\nK$.
\end{rmk}


\section{Formal and analytic classifications coincide}

This section is devoted to proving Theorem \ref{thm:Ecalle}.

Let $f:(\nK,0)\to(\nK,0)$ be a superattracting germ, and set $m=m(f)$, $d=d(f)$, $r_0=r_0(f)$ given by Definitions \ref{def:m}, \ref{def:d}, \ref{def:r}.

Suppose $\wt{f}:(\nK,0)\to(\nK,0)$ is a superattracting germ formally conjugate to $f$, and let $\Phi(x)=x\sum_n \phi_n x^n$ be the formal conjugacy between $f$ and $\wt{f}$.
We may assume $\phi_0=1$.

We want to show that $\Phi$ converges.
When $\nu_p(d)=0$, the assertion follows by Theorem \ref{thm:Beasy}.

Suppose $\nu_p(d)>0$, or equivalently $r_0 > 0$.
By Theorem \ref{thm:Bext}, we can suppose that $f$ is of the form
$$
f(x)=g(y)=C y^d \big(1+\eps(y)\big),
$$
where $\eps:(\nK,0) \to (\nK,0)$ is a convergent power series with $\eps(0) = 0$, and $\wt{f}$ is given by the truncation of $f$ at a suitable order $< p^m\big(d + \big\lfloor \frac{r_0p}{p-1}\big\rfloor + 1\big)$.

Write $C\big(1+\eps(y)\big)=\sum_n \eps_n y^n$.
Up to linear conjugacy, we can suppose $\abs{\eps_n} \leq 1$ for any $n$.
Recall that by definition, $\eps_{r_0} \neq 0$.

By the proof of Theorem \ref{thm:Bext} (see also Lemma \ref{lem:howJis}), the conjugacy relation $\Phi \circ f = \wt{f} \circ \Phi$ implies 
$$
r_0 \eps_{r_0} \phi_n^{p^m} = \lhs_{n+r_0} - \rhs^*_{n+r_0},
$$
for any $n$ big enough ($n > \left\lfloor \frac{p r_0}{p-1} \right\rfloor -r_0$).
Here $\rhs^*_{n+r_0}$ is defined as
\begin{equation}\label{eqn:Erhsstar}
\rhs^*_{n+r_0} = \sumaa{\substack{i \geq 0, I \in \nN^{d+i}\\ i + \abs{I}=n+r_0 \\ (i, I) \not \in \mc{I}}}{0.5cm} \wt{\eps}_i \phi_I^{p^m},
\end{equation}
where $\mc{I}=\{(r_0, I)\ |\ I=(0, \ldots, 0, n, 0, \ldots, 0)\}$.
Notice that for such big $n$, we have that $\lhs_{n+r_0}$ and $\rhs^*_{n+r_0}$ depend on $\phi_j$ only for $j < n$.

We now show that the sequence $\phi_n$ grows at most exponentially fast, which implies the result.
Set
$$
s_0 := \left\lfloor \frac{p r_0}{p-1} \right\rfloor +1-r_0 \geq 1, \qquad s_{h+1} := p s_h - r_0.
$$

Notice that the value of $s_0$ we picked has the property that if $n=pl-r_0 \geq s_0$, then $n > l \geq s_0$.

By solving the recurrence, we get
$$
s_h =s_0 p^h - r_0 \frac{p^h-1}{p-1}.
$$
The difference $k_h:=s_{h+1}- s_h$ is given by
$$
k_h= p^h (s_0(p-1) - r_0).
$$
Notice that $s_0 (p-1) > r_0 p -r_0 (p-1) = r_0$, hence $k_h \geq p^h \geq 1$ for any $h$.
Let $\gamma > 1$ be such that $\abs{\phi_n} \leq \gamma^n$ for $n \leq s_0$.

Set $\eta:= - p^{-m} \log_\gamma \abs{\eps_{r_0}}\geq 0$, and
$$
t_{h+1}:= p t_h + \eta, \qquad t_0:=s_0.
$$
We shall show that
$$
\abs{\phi_{n}} \leq \gamma^{s_0 + c (n-s_0)} \mbox{ for any } n \geq s_0, \qquad c:= \frac{s_0(p-1) + \eta}{s_0(p-1)-r_0}>1.  
$$
Up to increasing $\gamma$, we assume $\eta < c-1$.

More precisely, we show the following estimates.
\begin{prop}\label{prop:Eest}
Let $n \geq s_0$, and $h \in \nN, 0 \leq k < k_h$ be such that $n=s_h + k$.
Set
$$
\delta_h := \frac{c-1-\eta}{k_h}, \qquad c_n:=c_{h,k}:=t_h + k(c - \delta_h).
$$
Then for any $n \geq s_0$ we have
\begin{align}
\abs{\lhs_{n+r_0}} &\leq \gamma^{c_{n-1}}, \label{eqn:Eestlhs}\\ 
\abs{\rhs^*_{n+r_0}} &\leq \gamma^{p^m (c_n-\eta)}, \label{eqn:Eestrhs}\\
\abs{\phi_n} &\leq \gamma^{c_n}. \label{eqn:Eestphi}
\end{align}
\end{prop}
\begin{proof}
We prove these estimates by induction on $n$.

Set $c_n=n$ when $n \leq s_0$.
Notice that the sequence $c_n$ is (strictly) increasing.
Possibly increasing $\gamma$, we can suppose that the estimates hold for $n \leq s_0$, and get the basis of the induction.

Suppose the estimates \eqref{eqn:Eestlhs}, \eqref{eqn:Eestrhs} and \eqref{eqn:Eestphi} hold for any $n'<n$, we want to prove them for $n$.

Since $n\geq s_0$, we have that $\lhs_{n+r_0}$ depends only on $\phi_j$, $j<n$.
It follows that $\abs{\lhs_{n+r_0}} \leq \gamma^{c_{n-1}}$, and \eqref{eqn:Eestlhs} holds.

To obtain the remaining two estimates, we need the following properties for the sequence $(c_n)_n$.

\begin{lem}\label{lem:comparecnlinear}
For any $n=s_h+k \geq s_0$ with $0 \leq k \leq k_h-1$, we have
\begin{equation}\label{eqn:cnlinear}
c_n= s_0 + c(n-s_0) - k \delta_h,
\end{equation}
\end{lem}
\begin{proof}
Let us prove \eqref{eqn:cnlinear} when $k=0$, by induction on $h$.
For $h=0$ the equality trivially holds.
Suppose the equality holds for $h$, and let us prove it for $h+1$.
It suffices to show that $t_{h+1}-t_h=c_{s_{h+1}}-c_{s_h}=c k_h$.
But $t_{h+1}-t_h=p^h(s_0(p-1)+\eta) = p^h c (s_0(p-1)-r_0) = c k_h$.
Suppose now $n=s_h+k$ with $0 \leq k \leq k_h-1$.
Then
$$
c_n=t_h + k(c-\delta_h) = s_0 + c(s_h-s_0) + c k - k \delta_h = s_0 + c(n-s_0) -k \delta_h.
$$ 
\end{proof}
\begin{lem}\label{lem:cnsupadd}
The sequence $(c_n)$ satisfies the following properties.
\begin{enumerate}[(a)]
\item If $n = n'+n'' \leq s_0$, then $c_{n'}+c_{n''} = c_n$.
\item Suppose $n = n'+n'' > s_0$, $n'n'' \neq 0$, and set $n=s_h + k$ such that $0 \leq k < k_h$. Then $c_{n'} + c_{n''} + \delta_h \leq c_n -\eta$.
\item Suppose $n=pl-r$ with $l \geq s_0$. Then $pc_l \leq c_n-\eta$.
\end{enumerate}
\end{lem}
\begin{proof}
The property (a) trivially holds.
From Lemma \ref{lem:comparecnlinear}, for any $n=s_h+k \geq s_0$ with $0 \leq k \leq k_h-1$ we infer
\begin{equation}\label{eqn:lemcnsupaddest1}
s_0 + c(n-s_0) - (k_h-1) \delta_h \leq c_n \leq s_0 + c(n-s_0).
\end{equation}
If $1 \leq n \leq s_0$, then we get
\begin{equation}\label{eqn:lemcnsupaddest2}
c_n = n \leq c(n-1)+1 = s_0 + c(n-s_0) + (c-1)(s_0-1).
\end{equation}
Suppose we are in case (b), and let $h,k$ be such that $n=s_h+k$, $0 \leq k \leq k_h-1$.

First, suppose $n', n'' \leq s_0$.
Then we have
\begin{align*}
c_{n'}+c_{n''} + \delta_h +\eta =& n + \delta_h +\eta = s_0 + c(n-s_0) - (k_h-1) \delta_h + (c-1)(s_0-n) + k_h \delta_h + \eta \\
&\leq c_n + (c-1) (s_0-n) + k_h \delta_h + \eta \leq c_n,
\end{align*}
where the first inequality is given by \eqref{eqn:lemcnsupaddest1}, and the last inequality holds since $s_0-n \leq -1$ and $k_h \delta_h=c-1-\eta$ by definition.

Now suppose that $n'n''\neq 0$ and either $n'$ or $n''$ is $> s_0$.
We have
\begin{align*}
c_{n'} + c_{n''} + \delta_h + \eta \leq 2s_0 &+ c(n'+n''-2s_0) + (c-1)(s_0-1) + \delta_h + \eta\\
&= (1-c) + s_0 + c (n-s_0) + \delta_h +\eta \leq c_{n} + (1+\eta -c) + k_h \delta_h = c_n,
\end{align*}
where the first inequality is given by \eqref{eqn:lemcnsupaddest1} and \eqref{eqn:lemcnsupaddest2}, and the last inequality again by \eqref{eqn:lemcnsupaddest1}.

We now prove (c).
Write $l=s_h + k$, with $0 \leq k \leq k_h-1$.
Then $n=pl-r_0= ps_h +pk - r_0 = s_{h+1}+pk$. Notice that $0 \leq pk < k_{h+1}$.
Hence
$$
pc_l = p(t_h + k(c-\delta_h)) = \underbrace{pt_h + \eta}_{=t_{h+1}} + pk(c-\delta_{h+1}) - \eta +pk(\delta_{h+1}-\delta_h) = c_n - \eta - pk (\delta_h - \delta_{h+1}) \leq c_n - \eta.
$$
\end{proof}

We come back to the proof of Proposition \ref{prop:Eest}.
Consider the sum in \eqref{eqn:Erhsstar}.
If $i \geq r_0$, then for any dummy variable $I$ in the sum we have $\abs{I}=n+r_0-i \leq n$.
Recalling that if $i=r_0$ then $I \neq (0, \ldots, 0,n,0, \ldots, 0)$, by Lemma \ref{lem:cnsupadd}.(a,b) and the induction hypothesis we infer $\abs{\phi_I} \leq \gamma^{c_n-\eta}$.

Suppose $0 \leq i < r_0$ be such that $\eps_i \neq 0$.
By definition of $r_0$, we infer $\nu_p(i) \geq 1$.
If $\nu_p(n+r_0) = 0$, by Lemma \ref{lem:easy2} we have
$$
\sumaa{\substack{J \in \nN^{d+i}\\i+\abs{J}=n+r_0}}{0.3cm} \phi_J^{p^m} = 0,
$$
and we get \eqref{eqn:Eestrhs} in this case.

Suppose now $n+r_0=pl$ with $l \in \nN^*$.
Thanks to our choice of $s_0$, we have that $n \geq s_0$ implies $n > l \geq s_0$.
Suppose we have $I \in \nN^{(d+i)/p}$ with $\abs{I}=l$.
By Lemma \ref{lem:cnsupadd}.(c) and the induction hypothesis, we get $\abs{\phi_I^p} \leq \gamma^{pc_l} \leq \gamma^{c_n-\eta}$.
The estimate \eqref{eqn:Eestrhs} easily follows.

We now prove \eqref{eqn:Eestphi}. Since $n \geq s_0$, $\phi_n$ satisfies
$$
\phi_n^{p^m} = \frac{1}{r_0 \eps_{r_0}}\big(\lhs_{n+r_0}- \rhs^*_{n+r_0}\big).
$$
By \eqref{eqn:Eestlhs} we have $\abs{\lhs_{n+r_0}} \leq \gamma^{c_{n-1}}$.
Notice that $c_{n-1} \leq c_n - \eta$ if we pick $\eta$ small enough (i.e., $\gamma$ big enough).
Hence
$$
\abs{\phi_n}^{p^m} \leq \gamma^{p^m\eta} \left(\abs{\lhs_{n+r_0}} + \abs{\rhs^*_{n+r_0}}\right) \leq \gamma^{p^m (c_n - \eta + \eta)} = \gamma^{p^m c_n},
$$
from which \eqref{eqn:Eestphi} follows.
\end{proof}

By \eqref{eqn:Eestphi} it follows that there exists $\gamma \gg 1$ such that 
$$
\abs{\phi_n} \leq \gamma^{c_n} \leq \gamma^{s_0(1-c)} (\gamma^c)^n,
$$
where the last equality follows from Lemma \ref{lem:comparecnlinear}.
Hence the power series $\phi(x)$ converges when $\abs{x} < \gamma^{-c}$.

\begin{rmk}\label{rmk:Ecallegeneralization}
Theorem \ref{thm:Ecalle} hold over fields $\nK$ with milder hypotheses than being algebraically closed.
In fact, it holds over a field $\nK$ as far as the superattracting germs admit polynomial normal forms (of degree $< p^m\big(d + \big\lfloor \frac{r_0p}{p-1}\big\rfloor + 1\big)$).
It hold for example for any field $\nK$ closed under taking $p^m$-th roots (see Remark \ref{rmk:Bextgeneralization}).
\end{rmk}


\section{Superattracting germs in higher dimensions}

In this section we prove Theorem \ref{thm:monomial}, and give a few remarks on the local dynamics of superattracting germs in higher dimensions.
The notations used for the combinatorics in higher dimensions are the same explained in detail in the introduction and in the first section.

\begin{proof}[Proof of Theorem \ref{thm:monomial}]
Let $\ve{f}:(\nK^N, 0) \rightarrow (\nK^N,0)$ be a superattracting germ of the form
\begin{equation*}\tag{\ref*{eqn:Mstart}}
\ve{f}(\vxx)=\ve{C} \vxx^\mi{D} \big(\one + \eps(\vxx)\big).
\end{equation*}
Let $\ve{\wt{f}}:(\nK^N, 0) \rightarrow (\nK^N,0)$ be our candidate normal form
\begin{equation*}\tag{\ref*{eqn:Mnormalform}}
\ve{\wt{f}}(\vxx)=\ve{C} \vxx^\mi{D},
\end{equation*}
Let us consider a formal automorphism of the form
\begin{equation}\label{eqn:Mchangeofcoord}
\ve{\Phi}(\vxx)=\vxx \ve{\phi}(\vxx) = \big(x^1 \phi^1(\vxx), \ldots, x^N \phi^N(\vxx)\big),
\end{equation}
where $\vxx=(\xx^1, \ldots, \xx^n)$ and $\phi^j(\ve{0})=1$ for any $j=1, \ldots, N$.
We want to find a $\ve{\Phi}$ satisfying the conjugacy relation $\ve{\Phi} \circ \ve{f}= \ve{\wt{f}} \circ \ve{\Phi}$.
The conjugacy relation is equivalent to
$$
\big(\one + \ve{\eps}(\vxx)\big)\big(\ve{\phi} \circ \ve{f}(\vxx)\big) = \big(\ve{\phi}(\vxx)\big)^{\mi{D}}.
$$
For any $n \in \nN$, set
$$
\ve{\phi}_n(\vxx) = \prod_{k=1}^{n} \big(\one + \ve{\eps}\circ \ve{f}^{\comp k-1}(\vxx)\big)^{\mi{D}^{-k}}.
$$
Notice that $\ve{\phi}_n$ is well defined for any $n \in \nN$.
Indeed, since $\det \mi{D}$ and $p$ are coprime, for any $k$ we have that the entries of $\mi{D}^{-k}$ are of the form $a/b$, with $a,b \in \nZ$ and $\nu_p(a) \geq \nu_p(b)$.
By Remark \ref{rmk:rootwelldefined}, $\big(\one + \ve{\eps}\circ \ve{f}^{\comp k-1}(\vxx)\big)^{\mi{D}^{-k}}$ is a vector of well defined analytic germs for any $k$.

We want now to show that the sequence $\ve{\phi}_n$ converges to a suitable analytic germ $\ve{\phi}_\infty$, that will define the conjugacy we are looking for.
By Proposition \ref{prop:convergenceiseasy}, it suffices to show that $\norm{\ve{\eps} \circ \ve{f}^{\comp n}(\vxx)}$ tends to $0$.

Since $\ve{\eps}$ is analytic with $\ve{\eps}(\ve{0})=\ve{0}$, there exists $M \gg 0$ such that $\norm{\ve{\eps}(\vxx)}\leq M \norm{\vxx}$ for $\norm{\vxx} \ll 1$.
Since $\ve{f}$ is contracting, there exists $0<\Lambda < 1$ such that $\norm{\ve{f}^{\comp n}(\vxx)} \leq \Lambda^n \norm{\vxx}$ for $\norm{\vxx} \ll 1$ and any $k$.
It follows that
$$
\norm{\ve{\eps} \circ \ve{f}^{\comp n}(\vxx)} \leq M \Lambda^n \norm{\vxx} \to 0
$$
for $n \to \infty$ and $\norm{\vxx}$ small enough.
\end{proof}

\begin{rmk}\label{rmk:MwhataboutC}
In the previous theorem, the vector $C$ is invariant by change of coordinates of the form \eqref{eqn:Mchangeofcoord}.
This is clearly not the case for linear change of coordinates.
Since we want the change of coordinates to preserve the monomial normal form \eqref{eqn:Mnormalform}, we reduce ourselves to consider change of coordinates of the form $\vxx \mapsto \Delta \vxx$, where $\Delta$ is a diagonal matrix.
It is then easy to show that if $1$ is not an eigenvalue for $\mi{D}$, then there exists $\Delta$ such that the associated linear map conjugates $f$ to a germ of the form \eqref{eqn:Mnormalform} with $C=\one$.
This is not the case in general if $1$ is an eigenvalue for $\mi{D}$.
Indeed it can be easily shown that the moduli space up to conjugacy of such germs has dimension equal to the rank of $\mi{D}-\mi{\on{Id}}$.
\end{rmk}

\begin{rmk}\label{rmk:problemshighdim}
Using similar techniques, it is possible to extend some of the results in \cite{ruggiero:rigidgerms} over (normed) fields $\nK$ of characteristic $p>0$.
In particular, Theorem 2.7 in op.cit. still holds if we replace $\nC$ by any such $\nK$.
Moreover, Theorem 3.7 in op.cit. holds again if we replace $\nC$ by $\nK$, as far as $\det \mi{D}$ is coprime to $p$.

When $p$ divides $\det{\mi{D}}$, Theorem \ref{thm:monomial}, and Theorem 3.7 in op.cit., do not hold in general.
Indeed, the proof of these theorems uses the fact that the matrix $\mi{D}$ is invertible in $\nK$.
The study of this problem when $p$ divides $\det{\mi{D}}$ is much more complicated in higher dimensions, since we lose the natural total order on the coefficients of the map $\ve{f}$ developed in formal power series.

Notice that Theorem 2.7 in op.cit. gives the classification of contracting automorphisms when $r=p=0$, $s=d$ (with respect to the notations of \cite[Theorem 2.7]{ruggiero:rigidgerms}). Such a classification is well known under the name of \emph{Poincar\'e-Dulac} normal forms (see \cite{herman-yoccoz:smalldivisornonarchi}, and \cite{sternberg:localcontractionstheorempoincare,rosay-rudin:holomorphicmaps,berteloot:methodeschangementechelles} for the analogous problem in the complex setting).
This is in sharp contrast with the Poincar\'e-Dulac theorem for vector fields, where the study of resonances is much more intricate.
See \cite[Chapter 1]{ilyashenko-yakovenko:analDE} for an extensive presentation of Poincar\'e-Dulac normal forms and resonances for vector fields in the complex setting, and \cite[Part I]{herman-yoccoz:smalldivisornonarchi} for some results and remarks in the non-archimedean setting.
\end{rmk}

The results in \cite{ruggiero:rigidgerms} cited in Remark \ref{rmk:problemshighdim} are partial extensions in higher dimensions of the classification of contracting rigid germs given by Favre \cite{favre:rigidgerms} in dimension $2$. 

A \emph{rigid germ} is an analytic germ  $f:(\nC^2,0) \to (\nC^2,0)$ whose generalized critical set $C(f^\infty):=\bigcup_{n=1}^{\infty} C(f^n)$ has simple normal crossings and is forward $f$-invariant.

Favre and Jonsson (\cite{favre-jonsson:eigenval}) showed that any superattracting germ $f:(\nC^2,0) \to (\nC^2,0)$ is \emph{birationally} conjugate to a rigid germ (the condition of $f$ being superattracting is not necessary, see \cite{ruggiero:rigidification}).
Moreover, rigid germs and their normal forms have many applications for the study of a special class of non-K\"ahler compact complex surfaces, called Kato surfaces (see for example \cite{dloussky:phdthesis,nakamura:surfacesclass70curves,dloussky-oeljeklaus:vectorfieldsfoliationsoncptsurfacesclass70,dloussky-oeljeklaus-toma:surfacesclasse70champvecteurs2,toma:kahlerrankcpctcplxsurf}), and the study of the basin for attraction at infinity of suitable polynomial automorphisms, called H\'enon maps (see \cite{hubbard-oberstevorth:henonmappings1}, and \cite{friedland-milnor:dynpolyauto} for a precise description of the group of polynomial automorphisms in $\nC^2$).

Favre's classification provides \emph{polynomial} normal forms for contracting rigid germs.
Moreover, formal and analytic classifications coincide for \emph{superattracting} rigid germs.

These properties remain valid over any field $\nK$ of characteristic zero.
When $\nK$ has characteristic $p>0$, Favre's classification is still valid whenever $p$ and $\det\mi{D}$ are coprime, where $\mi{D}$ represents the action induced by $f$ on the fundamental group $\pi_1(\Delta^2 \setminus C(f^\infty))$, where $\Delta^2$ is a small polydisc centered at the origin (see \cite{favre:rigidgerms} for details).
The case when $p$ divides $\det{\mi{D}}$ still needs to be understood.
It is natural then to formulate the following 	questions.

\begin{quest}
Do there exist polynomial normal forms for contracting rigid germs $f:(\nK^2,0) \to (\nK^2,0)$, where $\nK$ is a (algebraically closed) field of positive characteristic?
\end{quest}

\begin{quest}
Do the formal and analytic classifications of superattracting rigid germs $f:(\nK^2,0) \to (\nK^2,0)$ coincide, when $\nK$ is a complete normed (algebraically closed) field of positive characteristic?
\end{quest}

\bibliographystyle{alpha}
\bibliography{biblio}

\end{document}